\documentclass[12pt,reqno]{amsart}

\usepackage{color}
\usepackage{amsmath, amsthm, amssymb}
\usepackage{amsfonts}
\usepackage[ansinew]{inputenc}
\usepackage[dvips]{epsfig}
\usepackage{graphicx}
\usepackage[english]{babel}
\usepackage{enumerate}
\usepackage{hyperref}
\theoremstyle{plain}
\newtheorem{thm}{Theorem}[section]

\newtheorem{conj}[thm]{Conjecture}
\newtheorem{question}[thm]{Question}
\newtheorem{lem}[thm]{Lemma}
\newtheorem{prop}[thm]{Proposition}

\theoremstyle{definition}
\newtheorem{defi}[thm]{Definition}

\theoremstyle{remark}
\newtheorem{rem}[thm]{Remark}

\numberwithin{equation}{section}

\newcommand{\de}{\partial}

\newcommand{\R}{\mathbb{R}}

\newcommand{\N}{\mathbb{N}}

\newcommand{\eps}{\varepsilon}

\newcommand{\average}{{\mathchoice {\kern1ex\vcenter{\hrule height.4pt
width 6pt depth0pt} \kern-9.7pt} {\kern1ex\vcenter{\hrule
height.4pt width 4.3pt depth0pt} \kern-7pt} {} {} }}

\def\R{\mathbb{R}}

\begin{document}

\title[On global solutions related to free boundary problems]{\hspace{-10mm}On global solutions to semilinear elliptic equations \hspace{-10mm}\vspace{1mm} related to the one-phase free boundary problem}

\author{Xavier Fern\'andez-Real}
\address{ETH Z\"urich, Department of Mathematics, R\"amistrasse 101, 8092 Z\"urich, Switzerland}
\email{xavierfe@math.ethz.ch}

\author{Xavier Ros-Oton}
\address{Universit\"at Z\"urich, Institut f\"ur Mathematik, Winterthurerstrasse, 8057 Z\"urich, Switzerland}
\email{xavier.ros-oton@math.uzh.ch}



\begin{abstract}
Motivated by its relation to models of flame propagation, we study globally Lipschitz solutions of $\Delta u=f(u)$ in $\R^n$, where $f$ is smooth, non-negative, with support in the interval $[0,1]$.
In such setting, any ``blow-down'' of the solution $u$ will converge to a global solution to the classical one-phase free boundary problem of Alt--Caffarelli.

In analogy to a famous theorem of Savin for the Allen--Cahn equation, we study here the 1D symmetry of solutions $u$ that are energy minimizers.
Our main result establishes that, in dimensions $n<6$, if $u$ is axially symmetric and stable then it is 1D.
\end{abstract}

\maketitle

\section{Introduction}\label{sec1}

We are interested in the study of certain elliptic equations of the type
\begin{equation}\label{pb}
\Delta u=f(u)\quad \textrm{in}\quad \R^n
\end{equation}
that are related to the classical one-phase free boundary problem \cite{AC81}.

\subsection{Brief description of the model}

Let us briefly introduce the model we have in mind.
Imagine in a domain $\Omega\subset\R^n$ we have a function $w_\varepsilon\geq0$ that minimizes the functional
\begin{equation}
\label{eq.Jeps}
J_\varepsilon(w):=\int_{\Omega}\left\{|\nabla w|^2+\Phi_\varepsilon(w)\right\}dx,
\end{equation}
where $\Phi_\varepsilon'=\beta_\varepsilon$, and $\beta_\varepsilon(t)=\frac{1}{\varepsilon}\beta(t/\varepsilon)$ is an approximation of the Dirac measure at 0, for some $\beta$ smooth, non-negative, with $\int\beta = 1$, and with compact support (see Figure~\ref{fig.0} for a representation of $\Phi_\eps$).
Notice that, when $\varepsilon\to0$, the energy changes discontinuously across the value $w=0$.

This type of singular perturbation problems (and their parabolic counterparts) are typical models for flame propagation \cite{CV95,BL82,We03,PY07}, and are studied in detail in the book of Caffarelli and Salsa \cite{CS05}.

In such combustion models, the sets $\{w_\varepsilon=0\}$, $\{0<w_\varepsilon\leq \varepsilon\}$, and $\{w_\varepsilon>\varepsilon\}$, represent the \emph{burnt}, \emph{reaction}, and \emph{unburnt} zones, respectively.

\begin{figure}
\centering
\includegraphics{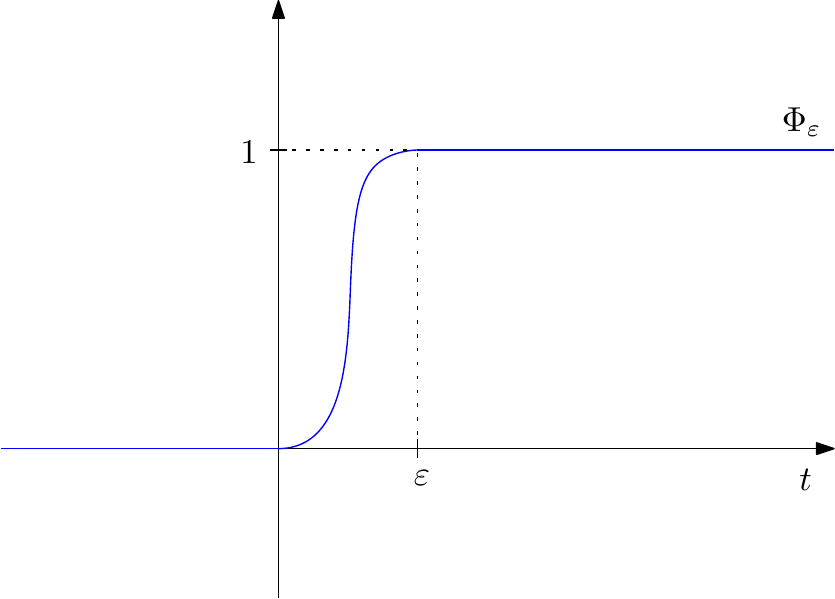}
\caption{Representation of $\Phi_\eps(t) = \int_0^t \beta_\eps(s)\, ds$. }
\label{fig.0}
\end{figure}

The transition between the two states (unburnt to burnt) occurs in a thin region of width comparable to $\varepsilon$.
If we try to understand such transition at length-scale $\varepsilon$, we need to dilate the picture by a factor $1/\varepsilon$.
Then, the rescaled solution $u(x):=\frac{1}{\varepsilon} w(\varepsilon x)$ minimizes the energy
\[J(u):=\int_{\frac{1}{\varepsilon}\Omega}\left\{|\nabla u|^2+\Phi(u)\right\}dx,\]
where $\Phi'=\beta$, and hence solves the elliptic equation
\begin{equation}\label{semilinear}
\Delta u=\textstyle{\frac12}\beta(u) \quad \textrm{in}\quad {\textstyle\frac{1}{\varepsilon}}\Omega
\end{equation}
in the large domain $\Omega/\varepsilon$.
On the other hand, if we let $\varepsilon\to0$ in the original domain~$\Omega$, the functions $w_\varepsilon$ converge to $w_0$, a minimizer of the one-phase problem, given by the functional
\begin{equation}
\label{eq.J0_0}
J_0(w) = \int_\Omega \left\{|\nabla w|^2+\chi_{\{w > 0\}}\right\}.
\end{equation}
Such minimizer $w_0$ fulfils
\begin{equation}\label{one-phase}
\Delta w_0=0\quad \textrm{in}\quad \{w_0>0\}\cap \Omega,
\end{equation}
\begin{equation}\label{one-phase2}
|\nabla w_0|=1\quad \textrm{on}\quad \partial\{w_0>0\}\cap \Omega.
\end{equation}
Therefore, there is a deep connection between the semilinear elliptic equation \eqref{semilinear} and the one-phase free boundary problem \eqref{one-phase}-\eqref{one-phase2}.
In fact, it is easy to see that any ``blow-down'' of a global solution to \eqref{semilinear} will converge to a global non-negative solution to the free boundary problem \eqref{one-phase}-\eqref{one-phase2}.

The solution to \eqref{semilinear} obtained by this procedure will \emph{grow linearly} at infinity.
By classical uniform Lipschitz estimates for this problem (see \cite{CS05}), we deduce that any such solution $u$ will be globally Lipschitz.

Thus, after letting $\varepsilon\to0$ in \eqref{semilinear}, we want to understand solutions of \eqref{pb}, with 
\[f = \textstyle{\frac12} \beta,\]
 and $\beta$ and $u$ satisfying
\begin{itemize}
\item[(A1)] $\beta$ is non-negative and $C^1$, $\textrm{supp}\,\beta=[0,1]$, and $\int_0^1\beta=1$.
\item[(A2)] $u\geq0$ satisfies $\|\nabla u\|_{L^\infty(\R^n)}<\infty$.
\end{itemize}

Let us see next what are the main questions in this context.

\subsection{Some natural open problems}

A natural notion in this context is the following.

\begin{defi}
Assume (A1)-(A2). We say that $u$ is a \emph{minimizer} of the energy $J$ in $\R^n$ when it minimizes
\[J(u,E)=\int_{E}\left\{|\nabla u|^2+\Phi(u)\right\}dx\]
for every compact set $E\subset\R^n$, i.e., when $J(u,E)\leq J(u+v,E)$ for every $v\in H^1_0(E)$. Here, $\Phi'=\beta$.
\end{defi}

Given the relation between our problem \eqref{pb} and the one-phase free boundary problem \eqref{one-phase}-\eqref{one-phase2}, and because of the known results on the regularity of free boundaries in such context \cite{JS15,CJK04}, it seems natural to conjecture the following:

\begin{conj}\label{conj1}
Assume (A1)-(A2), and let $u$ be any minimizer of the energy $J$ in $\R^n$, with $n\leq4$.
Then, $u$ is one-dimensional.
\end{conj}

In $\R^2$, such conjecture follows from well known results; see for example \cite[Theorem 2.1]{FV08}.
Thus, the important cases that remain open are $\R^3$ and $\R^4$.

The condition on the dimension, $n\le 4$, is based on the fact that minimal cones for the one-phase problem have been shown to be hyperplanes up to $n= 4$. This, however, is currently not known to be optimal, and in fact, Caffarelli, Jerison, and Kenig conjectured in \cite{CJK04} that this holds up to dimension $n= 6$ (there are counterexamples in dimension $n = 7$, as seen in \cite{DJ09}). Thus, linking both conjectures together, one could extend the statement of Conjecture~\ref{conj1} up to dimension $n \le 6$.

\begin{rem}
Notice that there is a very strong analogy between Conjecture \ref{conj1} and a famous theorem of Savin for the Allen--Cahn equation (\cite{Sa09}, see also \cite{DG78,AC00}).
Indeed, when the energy $J_\varepsilon$ above is replaced by
\[\int_\Omega \left\{\frac{\varepsilon}{2}|\nabla u|^2+W(u)\right\}dx,\]
where $W$ is a \emph{double-well} potential, then the above procedure yields solutions to the semilinear elliptic equation \eqref{pb}, with $f$ \emph{bistable}.
The most typical example is the Allen--Cahn equation $-\Delta u=u-u^3$ in $\R^n$, arising in the study of phase transitions.
An important difference is that in the Allen--Cahn case the solution $u$ obtained by this procedure is \emph{globally bounded} in $\R^n$, while ours grows linearly at infinity.
\end{rem}

More generally, one may expect Conjecture \ref{conj1} to hold even for {stable} solutions to \eqref{pb}, defined as follows.

\begin{defi}
Assume (A1)-(A2). We say that $u$ is a \emph{stable} solution of \eqref{pb} if the second variation of energy at $u$ is non-negative, i.e.,
\begin{equation}
\label{eq.stabcond_2}
\int_{\R^n}\left\{|\nabla \xi|^2+\textstyle{\frac12}\beta'(u)\xi^2\right\}dx\geq0
\end{equation}
for every $\xi\in C^\infty_c(\R^n)$.
\end{defi}

A second natural problem (c.f. \cite{DJ09}) is the following:

\begin{conj}
Given $\beta$ satisfying (A1), there exists a globally Lipschitz minimizer of the energy $J$ in $\R^7$ which is not 1D.
\end{conj}

As already mentioned, for the remaining dimensions, $n=5$ and $n=6$, the question is not even understood in case of the one-phase problem.

We think that all these are difficult and interesting questions to be studied.

\subsection{Axially symmetric solutions}

Here, we plan to study Conjecture \ref{conj1} for a special class of solutions with symmetries.

Notice that in case of the Allen--Cahn equation, the natural candidates for non-1D solutions are those with symmetry of ``double revolution''; see \cite{CT09,LWW16}.
This is because in minimal surfaces the simplest counterexample to regularity is the Simons cone, which has this symmetry.

In the one-phase free boundary problem, instead, the natural nontrivial solutions are \emph{axially symmetric}; see \cite{CJK04,DJ09}. Namely, solutions $u:\R^n\to \R$, depending only on $s := \sqrt{x_1^2+\dots+x_{n-1}^2}$ and $t := x_n$
\[
u(x_1,\dots,x_n) = u(|x'|, x_n) =u(s, t),
\]
where $x'= (x_1,\dots,x_{n-1})\in \R^{n-1}$.
 
In our setting \eqref{pb}, we have the following natural:

\begin{question}
\label{question.1}
Assume (A1)-(A2). Can one prove that all axially symmetric solutions to \eqref{pb} are either \emph{1D} or \emph{unstable} in dimensions $n\leq6$?
\end{question}

The aim of this paper is to investigate this question.

\subsection{Our results}

Following, we present the two main theorems of this paper. The first result answers affirmatively Question~\ref{question.1} up to dimensions $n \le 5$; namely, it says that any axially symmetric solution to \eqref{pb} is either unstable or one-dimensional. We would like to remark that, in the proof, the dimension can be taken up to $n < 6$, but that the limiting case $ n = 6$ is not achieved (and thus, remains an open problem). 

\begin{thm}
\label{thm.1stabprop}
Let $\beta$ satisfy (A1). Let $u\in {\rm Lip}(\R^n)$ be a stable axially-symmetric solution to 
\[
\Delta u = \textstyle{\frac12}\beta(u),\quad\textrm{ in } \R^n.
\]
Then, if $n \le 5$, $u$ is one-dimensional. 
\end{thm}

\begin{rem}
It is worth mentioning that the conditions on $\beta$ can be relaxed considerably. Indeed, the proof of the theorem only uses that $\beta\in C^1$, with no extra assumption. 

Nonetheless, the fact that $u\in {\rm Lip}(\R^n)$ solves the problem already poses some restrictions on $\beta$. For instance, it would be enough to assume that $\beta$ is compactly supported in $\R^n$, instead of the more restrictive condition that $\beta\ge 0$ and ${\rm supp }\, \beta = [0,1]$. That is, $\beta$ could have a non-connected support and be negative somewhere. Under these assumptions, however, it is no longer true that one can prove the $\Gamma$-convergence to the one-phase problem, and studying (and classifying) one-dimensional solutions becomes a more difficult task. 
\end{rem}

The second result studies the same kind of solutions for the one-phase free boundary problem. In this case, Liu, Wang, and Wei, constructed in \cite{LWW17} smooth axially-symmetric solutions of ``catenoid'' type. 

\begin{thm}
\label{thm.1stab1phase}
Let $u\in{\rm Lip}(\R^n)$ be a stable, non-negative, axially-symmetric solution to the one-phase problem. That is, $u$ is a stable critical point of the minimizer 
\begin{equation}
\label{eq.onephasefunctional}
J(u, B) = \int_B\left\{ |\nabla u|^2+\chi_{\{u > 0\}}\right\}
\end{equation}
for any ball $B\subset\R^n$. Suppose also that $\partial\{u > 0\}$ is smooth. Then, if $n\le 5$, $u$ is one-dimensional.  
\end{thm}

In particular, this result yields that the axially symmetric solutions to the one-phase problem constructed in \cite{LWW17} are unstable if $n \le 5$. 

\subsection{Organization of the paper}

The paper is organized as follows. In Section~\ref{sec.2} we prove some basic properties of solutions. In particular, we show that the Lipschitz condition can be relaxed to a linear growth at infinity of solutions, we prove the $\Gamma$-convergence of the functionals in \eqref{eq.Jeps} to the one-phase energy functional, and we make a general discussion about one-dimensional solutions to \eqref{pb}. In Section~\ref{sec.3} we then prove Theorem~\ref{thm.1stabprop}, and finally, in Section~\ref{sec.4} we prove Theorem~\ref{thm.1stab1phase}. 
\vspace*{0.5cm}

{\it Acknowledgement:} The first author is supported by the ERC Grant ``Regularity and Stability in Partial Differential Equations (RSPDE)''.

\section{Basic properties of solutions}
\label{sec.2}

Let us begin by showing some basic properties of solutions to our problems, and of the energy functionals involved in them. 

\subsection{Linear growth implies globally Lipschitz}

Let us show that the globally Lipschitz conditions can be weakened to linear growth at infinity. 

\begin{prop} 
Let $u$ be non-negative, be a solution to \eqref{pb} with $f = \frac12\beta$. Suppose that 
\begin{equation}
\label{eq.linear}
\|u\|_{L^\infty(B_R)} \le C_0 R,\quad\textrm{ for all } R \ge 1.
\end{equation}
Then 
\begin{equation}
\label{eq.lincons}
\|\nabla u \|_{L^\infty(\R^n)} \le C C_0,
\end{equation}
for some $C$ depending only on $\beta$ and $n$. That is, $u$ is globally Lipschitz.
\end{prop} 
\begin{proof}
The proof of this result is straight-forward, using the results from \cite{CS05}. Indeed, define 
\[
u_\eps(x ) = \eps u\left(\frac{x}{\eps}\right). 
\]
Then, 
\[
\Delta u_\eps = \frac{1}{2\eps} \beta \left(\frac{u_\eps}{\eps}\right) = \frac12\beta_\eps(u_\eps),
\]
and putting $R = \eps^{-1}$ in \eqref{eq.linear} we have that
\[
\|u_\eps\|_{L^\infty(B_{1})}\le C_0.
\]
Now, thanks to \cite[Theorem 1.2]{CS05} (notice that the authors only use that $\Delta u = \frac12\beta(u)$ holds, rather than the minimality of the solution), 
\[
\|\nabla u_\eps\|_{L^\infty(B_{1/2})}\le C C_0,
\]
for some $C$ depending only on $\beta$ and $n$. Putting it in terms of $u$ we reach the desired result, \eqref{eq.lincons}. 
\end{proof}

\subsection{$\Gamma$-convergence to the one-phase problem}
Let us show the $\Gamma$-convergence in $W^{1, 2}(\Omega)$ of the functionals $J_\eps$ defined in \eqref{eq.Jeps}, to $J_0$, the energy functional associated to the one-phase problem, \eqref{eq.J0_0}.  That is, we show 
\[
J_\eps\xrightarrow{\Gamma} J_0,\qquad\textrm{ as }\quad\eps\downarrow 0,
\]
where 
\[
J_\varepsilon(w):=\int_{\Omega}\left\{|\nabla w|^2+\Phi_\varepsilon(w)\right\}dx\quad\textrm{ and }\quad J_0(u) = \int_\Omega \left\{|\nabla u|^2+\chi_{\{u > 0\}}\right\},
\]
and where $\chi_{\{u >  0\}}$ denotes the characteristic function of the set $\{u > 0 \}$. 

\begin{prop}
Let $J_\eps$ and $J_0$ as in \eqref{eq.Jeps}-\eqref{eq.J0_0}, and let (A1) hold. Then,  $J_\eps\xrightarrow{\Gamma} J_0$ in $W^{1,2}(\Omega)$ as $\eps\downarrow 0$.
\end{prop}
\begin{proof}
It is enough to show that:
\begin{enumerate}[(i)]
\item For every $u \in W^{1,2}(\Omega)$ and for every sequence $u_\eps$ such that $u_\eps \rightarrow u$ in $W^{1,2}(\Omega)$ as $\eps\downarrow 0$, then 
\[
\liminf_{\eps\downarrow 0} J_\eps(u_\eps) \ge J_0(u).
\]
\item For every $u \in W^{1,2}(B)$ there exists a sequence $u_\eps$ with $u_\eps \rightarrow u$ in $W^{1,2}(\Omega)$ as $\eps\downarrow 0$ such that
\[
\limsup_{\eps\downarrow 0} J_\eps(u_\eps) \le J_0(u).
\]
\end{enumerate}

Let us start by showing (i). The term involving the gradient is clear, since $u_\eps \rightarrow u$ in $W^{1,2}(\Omega)$. Thus, we have to show that 
\[
\liminf_{\eps\downarrow 0} \int_\Omega \Phi_\eps(u_\eps) \ge \int_\Omega \chi_{\{u > 0\}}. 
\]
Suppose that it is not true, so that there exists a sequence $(\eps_k)_{k\in \N}$, $\eps_k\downarrow 0$ as $k\to \infty$, such that 
\[
\lim_{k\to \infty} \int_\Omega \Phi_{\eps_k}(u_{\eps_k}) < \int_\Omega \chi_{\{u > 0\}}. 
\]

Since $u_{\eps_k} \rightarrow u$ in $W^{1,2}(\Omega)$, in particular, there exists a further subsequence $\eps_{k_j}$ such that $u_{\eps_{k_j}}$ converges to $u$ pointwise almost everywhere. In particular, since $\Phi \ge 0$, we have that
\[
\lim_{j\to \infty} \int_{\Omega\cap \{u > 0\}} \Phi_{\eps_{k_j}}(u_{\eps_{k_j}}) < \int_\Omega \chi_{\{u > 0\}}. 
\]
Notice, however, that for almost every $x\in \{u > 0\}$, $\Phi_{\eps_{k_j}}(u_{\eps_{k_j}})(x)\to 1$ (since $\lim_{y\to \infty}\Phi(y) = 1$ and $u_{\eps_{k_j}}(x)\to u(x) > 0$). In particular, by Fatou's lemma, we have 
 \[
\liminf_{j\to \infty} \int_{ \Omega\cap \{u > 0\}} \Phi_{\eps_{k_j}}(u_{\eps_{k_j}}) \ge \int_\Omega \chi_{\{u > 0\}},
\]
a contradiction. 

In order to show (ii) notice that $\Phi(x) \equiv 0$ for $x\in (-\infty,0]$, and that $\Phi' \ge 0$ everywhere. Then, take $u_\eps = u$ for all $\eps$, and notice that $\Phi_\eps (u_\eps ) = \Phi_\eps(u) \le \chi_{\{u \ge 0\}}$, so that we are done. 
\end{proof}

\begin{rem}
Notice that in the previous proof we are not using the fact that $\beta$ is smooth and has support in the interval $[0, 1]$. In particular, the previous lemma still holds true if $\beta \ge 0$, $\int_\R \beta = 1$, and has support in $[0, \infty)$. 
\end{rem}

\subsection{Existence and uniqueness of 1D solutions}

Let us make a description of the one dimensional solutions to the semilinear problem \eqref{pb} in the case (A1) holds and, in addition, 
\[
\beta > 0\quad\textrm{ in }\quad (0, 1).
\]

Let $u: \R^n\to \R$ be a one-dimensional solution of 
\[
\Delta u =\textstyle{\frac12} \beta(u),\quad\textrm{ in } \R^n.
\]
That is, $u  = u(x_1)$ and therefore, $u$ (understood as a function with one dimensional domain) solves the ODE
\[
u''(x_1) = \textstyle{\frac12}\beta(u (x_1)).
\]

By Picard's theorem, it is enough to prescribe $u$ and $u'$ at a point $x_0$ to get existence and uniqueness of solution.

Let us make a qualitative study of how the solution $u$ looks like. 
Since $\beta \ge 0$, in particular, $u $ is convex. Thus, either $u\equiv c$ for some constant $c\in \R\setminus (0, 1)$ --- recall $\beta  >0 $ in $(0, 1)$ --- or $u(x_1) \to +\infty$ as $x_1\to \infty$ or $x_1\to -\infty$. 

Let us assume that $\lim_{x_1\to \infty} u(x_1 ) = +\infty$, the other case follows by the symmetry of the problem. Notice that, since $\beta(t) \equiv 0$ for $t\ge 1$, this means that $u(x_1) = a x_1$ for some $a > 0$ whenever $u(x_1) \ge 1$. In particular, there exists a point $\alpha\in \R$, such that $u(\alpha) = 1$, and $u'(\alpha) = a$. After a translation, let us assume that $ \alpha = 1$, so that $u(1) = 1$, $u'(1) = a$. Now, for any $a > 0$ we reach a different solution to our problem. 

Now, since $u$ is convex, there exists some $p\in \R\cup\{-\infty\}$ such that $u'\ge 0$ in $(p, +\infty)$, and $u'\le 0$ otherwise. Notice that $p\neq +\infty$ from the assumption that $\lim_{x_1\to \infty} u(x_1) = +\infty$. We separate different cases according to the value of $p$.

\begin{figure}
\centering
\includegraphics[scale = 1]{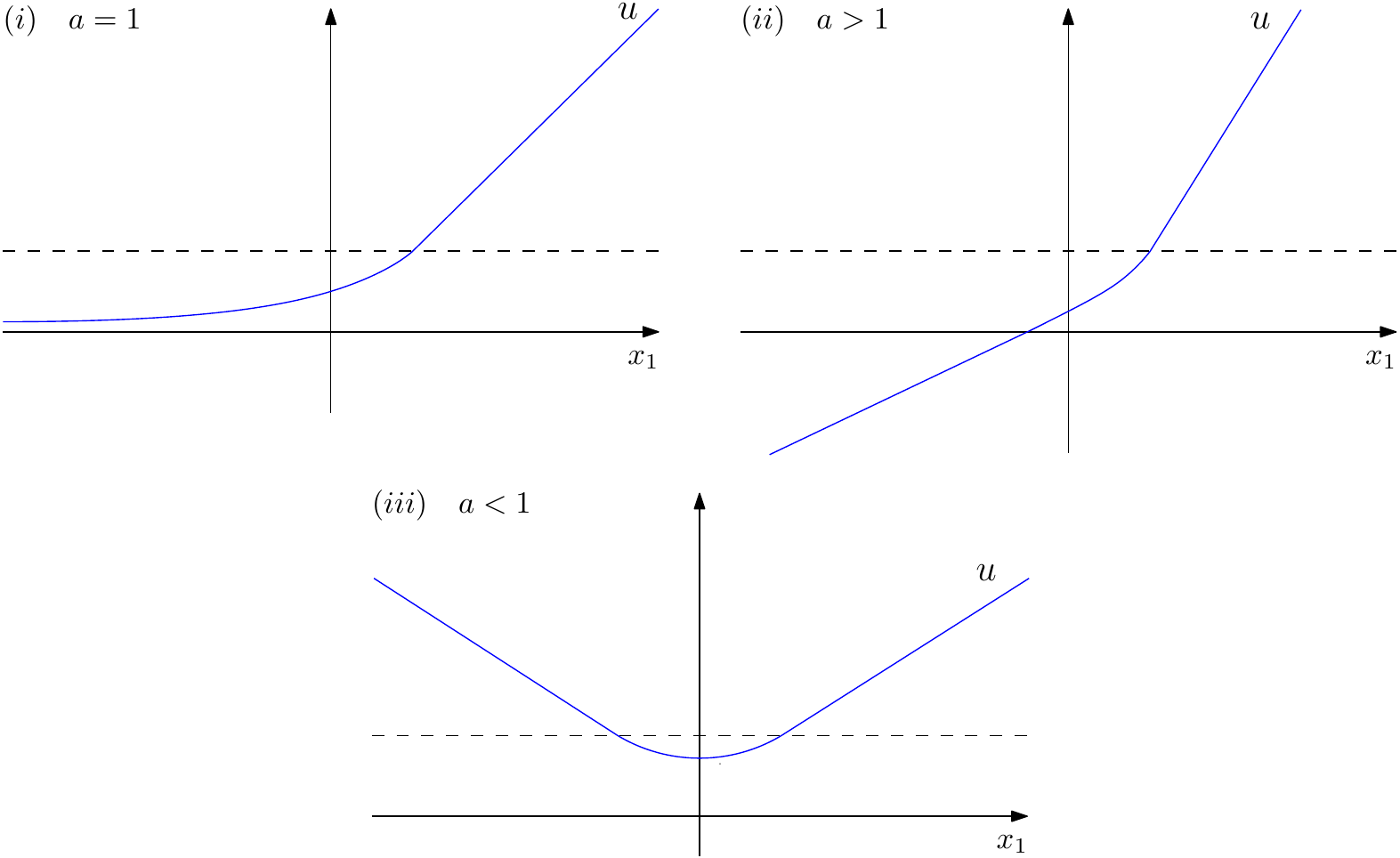}
\caption{Representation of the cases (i) $a > 1$, (ii) $a = 1$, and (iii) $ a < 1$. }
\label{fig.1}
\end{figure}

If $p = -\infty$, then $u'\ge 0$ in $\R$, it is a monotone nondecreasing function. Then, since $u$ is convex, we must necessarily have that $u ' > 0$ (otherwise, we would have that $u(x_1)$ is constant by convexity and uniqueness of our problem). In particular, $u(-\infty)$ is well defined (it could be $-\infty$), and since $u ''\ge 0$, $u'(-\infty)$ is also well defined. In fact, by strict convexity (recall that $\beta > 0$ in $(0, 1)$),  using and $\beta \equiv 0 $ outside of $(0, 1)$, we obtain that either $u(-\infty) = 0$ or $u(-\infty) = -\infty$, and respectively $u'(-\infty) = 0$ and $u'(-\infty)  =b$ for some $b > 0$. We have that
\[
a^2-|u'(-\infty)|^2 = \int_{-\infty}^1 (u'^2)' = \int_{-\infty}^1 u' \beta(u) = \int_{u(-\infty)}^1\beta(t) \,dt = 1,
\]
where in the last equality we are using that $\int_0^1 \beta = 1$, and $u(-\infty)\le 0$. Thus, we have that if the solution is monotone ($p =-\infty$),
\begin{enumerate}
\item[(i)] if $a > 1$, then the solution grows linearly at $+\infty$ with slope $a>0$, and linearly at $-\infty$ with slope $b>0$ such that $a^2-b^2 = 1$; with a smooth transition in between. Notice that the blow-down is a solution to the two-phase problem, consisting in two planes intersecting at 0 whose slopes satisfy the relation $a^2-b^2 = 1$ . See Figure~\ref{fig.1} (i). 
\item[(ii)] if $a  = 1$, then the solution grows linearly as $x_1\to\infty$ with slope 1, and goes to $0$ as $x_1\to -\infty$. The blow-down converges to a solution to the one-phase problem, given by $\max\{0, x_1\}$. See Figure~\ref{fig.1} (ii). 
\end{enumerate}

On the other hand, if $p>-\infty$, the solution is not monotone. After a translation, we can assume that $p = 0$, so that $u'(0) = 0$. Since the solution is not constant, and $\beta > 0$ on $(0, 1)$, this immediately yields that $u(0 ) = y_0\in (0, 1)$, and from the convexity of $u$ and $\beta(y_0) > 0$ we have that $u'(x_1) > 0$ for $x_1 > 0$. Thus, 
\[
a^2 = \int_0^1 (u'^2)' = \int_0^1 u' \beta(u) = \int_{y_0}^1\beta(t) dt < 1,
\]
and so $a < 1$. In this case, by symmetry the solution is even with respect to $x_1 = 0$ and we have that 
\begin{enumerate}
\item[(iii)] if $0<a < 1$, then the solution grows linearly as $x_1 \to \infty $ as $a x_1$, and the solution is also linear as $x_1\to -\infty$, but as $-a x_1$. The blow-down converges to two planes with slopes of opposite signs and absolute value smaller than~1. See Figure~\ref{fig.1} (iii). 
\end{enumerate}

Finally, if we considered the case $\lim_{x_1\to -\infty} u (x_1) = +\infty$ we would obtain the even reflections of the previous solutions, with $p \in \R\cup\{ +\infty\}$.

Thus, when $\beta > 0$ in $(0, 1)$, the previous solutions comprise all possible solutions to our problem, up to even reflections and translations. 

In general, under the assumption (A1), we have the following result about non-negative 1D solutions to our problem.

\begin{prop}
\label{prop.1D}
For any $\beta$ such that (A1) holds there exists a unique non-negative, increasing, non-constant, smooth solution $u$ (up to translations) to $u''= \frac12\beta(u)$ such that $u \to 0$ at $-\infty$.
\end{prop}
\begin{proof}
Notice that from the discussion above we have covered the case $\beta > 0$ in $(0, 1)$; noticing that the solution must tend to $\infty$ as $x_1\to \infty$ by convexity. The same reasoning as before also works if ${\rm supp}~\beta = [0, 1]$ by noticing that, in particular, $\int_\eps^1 \beta(t)\,dt < 1$ for all $\eps > 0$.
\end{proof}

\begin{rem}
Actually, such unique solution fulfils in addition $u'(x_1) = 1$ if $u(x_1) \ge 1$ and $\lim_{x_1\to -\infty} u'''(x_1)/u'(x_1) = 0$.

Indeed, differentiating $u''(x_1) =\frac12 \beta(u(x_1))$ we get $u'''(x_1)/u'(x_1) = \frac12\beta'(u(x_1))$. Now, letting $x_1\to -\infty$ we reach that $\lim_{x_1\to -\infty} u'''(x_1)/u'(x_1) = \frac12\beta'(0) = 0$. That is, the limiting condition of $u'''/u'$ comes from the $C^1$ regularity of $\beta$ at 0. 

Conversely, given a non-negative, increasing, convex, smooth function $v = v(x_1)$ such that $\lim_{x_1\to -\infty} v(x_1) = 0$, $v'(x_1) = 1$ if $v(x_1) \ge 1$, $v''(x_1) >0$ if $v(x_1) < 1$, and 
\begin{equation}
\label{eq.3rdder}
\lim_{x_1\to -\infty} v'''(x_1)/v'(x_1) = 0,
\end{equation}
then there exists a unique $\beta$ fulfilling (A1) such that $v'' = \frac12\beta(v)$.

Indeed, take $\beta(t) := 2v''(v^{-1}(t))$ if $t > 0$, $\beta \equiv 0$ otherwise, so that we obviously have $\frac12\beta(v(x_1)) = v''(x_1)$. Then $\beta \ge 0$, and ${\rm supp}~\beta = [0, 1]$ from the strict convexity of $v$ whenever $v(x_1) < 1$. On the other hand, $\beta(t) \equiv 0$ if $t \le 0$ or $t \ge 0$ and is smooth if $t > 0$ (since $v^{-1}$ is smooth and well-defined for $t > 0$). We just have to check the $C^1$ condition at $t = 0$ and the integrability condition.

Let us start by showing that $\beta$ and $\beta'$ are continuous at 0. Indeed, $\beta$ is continuous at 0 because $\lim_{x_1\to -\infty} v''(x_1) = 0$ and $v^{-1}(t) \to -\infty$ as $t\downarrow 0$. On the other hand, $\beta'(t) =2 v'''(v^{-1}(t)) / v'(v^{-1}(t)) \to 0$ as $t\downarrow 0$ by hypothesis, so that $\beta\in C^1(\R)$. 

Finally, we check the integrability condition. We know that $\lim_{x_1\to -\infty} v'(x_1) = 0$ and, following as above, 
\[
1 = (u'(+\infty))^2-(u'(-\infty))^2 = \int_{-\infty}^\infty (u'^2)' = \int_{-\infty}^\infty u'\beta(u) = \int_0^1\beta(t)\,dt
\]
as we wanted to see.

Finally, let us mention that the previous remarks are also true without the assumption \eqref{eq.3rdder} but then we might get a $\beta$ that is not smooth at the origin. That is not a problem, since Proposition~\ref{prop.1D} still holds removing the smoothness of $\beta$ at the origin (although then we might have $u \equiv 0$ in $(-\infty, p)$ for some $p\in \R$; that is, the solution stabilizes at constant 0). 
\end{rem}

\section{Proof of Theorem~\ref{thm.1stabprop}}
\label{sec.3}

In this section we prove that global Lipschitz stable axially-symmetric solutions to $\Delta u = \frac12\beta(u)$ are one-dimensional in dimensions $n \le 5$. 

\begin{proof}[Proof of Theorem~\ref{thm.1stabprop}]
If $n = 2$, this follows from a known result in \cite[Theorem 2.1]{FV08}. Thus, we assume $n \ge 3$. We use some ideas from  \cite{CC04, CR13}.

Let us begin by doing a sketch of the proof, which will be formalized below: 

The main idea is to take $\xi = u_s\eta$ in the stability condition, to reach 
\[
(n-2)\int_{\R^n} u_s^2\eta^2s^{-2} \le \int_{\R^n} u_s^2|\nabla \eta|^2.
\]
Then, formally taking $\eta = s^{-\alpha}$ with $\alpha > 0$, we deduce
\[
(n-2-\alpha^2)\int_{\R^n} u_s^2 s^{-2\alpha-2}\le 0,
\]
so that $u_s \equiv 0$ in $\R^n$. As we will see, in the previous computations we have to control errors, use appropriate test functions, and make sure everything remains integrable throughout the proof. 

Let us begin the formal proof. We have that $u(x_1,\dots,x_n) = u(s, t)$, where $s = \sqrt{x_1^2+\dots+x_{n-1}^2}$ and $t = x_n$; that is, $u$ is axially symmetric. 

For any stable solution to 
\[
\Delta u = \textstyle{\frac12}\beta(u)\quad\textrm{ in } \R^n,
\]
and $\beta\in C^1$, then, for any $c \in (H^{2}_{\rm loc}\cap L^\infty)(\R^n\setminus \{s = 0\})$, using the stability condition \eqref{eq.stabcond_2} with $\xi = c\eta$ and integrating by parts one gets  
\begin{equation}
\label{eq.stab}
\int_{\R^n} c \left\{ \Delta c - \frac12\beta'(u) c\right\}\eta^2 \, dx \le \int_{\R^n} c^2 |\nabla \eta|^2\, dx,
\end{equation}
for all $\eta \in(H^1\cap L^\infty)(\R^n)$ with compact support in $\R^n\setminus \{s =0\}$. We want to use \eqref{eq.stab} with $c = u_s$ (which is not smooth for $s = 0$), and $\eta \in(H^1\cap L^\infty)(\R^n)$.

On the other hand, noticing that 
\begin{equation}
\label{eq.Lst}
\Delta u = u_{ss}+  \frac{n-2}{s} u_s + u_{tt},
\end{equation}
then we can differentiate $\Delta u = \frac12\beta(u)$ with respect to $s$ to  obtain 
\[
\Delta u_s - (n-2)\frac{u_s}{s^2} = \textstyle{\frac12}\beta'(u) \, u_s.
\]

Thus, putting $c = u_s$ in \eqref{eq.stab} (notice that $u_s \in (H^{2}_{\rm loc}\cap L^\infty)(\R^n\setminus \{s = 0\})$ by interior estimates) we reach 
\begin{equation}
\label{eq.stab2}
(n-2)\int_{\R^n} u_s^2\eta^2s^{-2}\, dx \le \int_{\R^n} u_s^2|\nabla \eta|^2\, dx
\end{equation}
for all $\eta \in(H^1\cap L^\infty)(\R^n)$ with compact support in $\R^n\setminus \{s =0\}$. Suppose, now, that $\eta \in(H^1\cap L^\infty)(\R^n)$ with compact support in $\R^n$, and take $\eta\zeta(s/\eps)$ as a test function in \eqref{eq.stab2}, where $\zeta(s) \in C^\infty(\R^+)$, $\zeta \equiv 1$ for $s \ge 1$, $\zeta' \ge 0$, $\zeta \equiv 0$ for $s\in [0, 1/2]$. By letting $\eps\downarrow 0$, and using that $|u_s|\to 0$ as $s\downarrow 0$ by regularity and axial symmetry, one reaches that \eqref{eq.stab2} also holds for $\eta \in(H^1\cap L^\infty)(\R^n)$ with compact support in $\R^n$. 

Now, for any $\eps > 0$ and $R > 1$ define $\eta_{\eps, R}$ as 
\[
  \eta_{\eps, R} = \left\{ \begin{array}{ll}
  s^{-\alpha}\rho_R & \textrm{ if } s > \eps\\
  \eps^{-\alpha}\rho_R  & \textrm{ if } s \le \eps,
  \end{array}\right.
\]
where $\rho_R \in C^\infty(\R^n)$ $\rho_R\ge 0$, is a smooth non-negative function such that
\[
\rho_{R} = \left\{ \begin{array}{ll}
  1 & \textrm{ in } B_R\\
  0 & \textrm{ in } \R^n \setminus B_{2R}.
  \end{array}\right.
\]

Then, 
\[
  |\nabla \eta_{\eps, R}|^2 \le \left\{ \begin{array}{ll}
  \alpha^2 s^{-2\alpha-2} \rho_R^2 & \textrm{ in } B_R\cap \{s > \eps\}\\
    \alpha^2 s^{-2\alpha-2} \rho_R^2 + s^{-2\alpha}|\nabla \rho_R|^2 & \textrm{ in } B_{2R}\setminus B_R\cap \{s > \eps\}\\
  \eps^{-2\alpha}|\nabla \rho_R|^2  & \textrm{ if } s \le \eps.
  \end{array}\right.
\]

Putting $\eta_{\eps, R}$ in the left hand side of \eqref{eq.stab2} we reach,
\[
(n-2)\int_{\R^n} u_s^2\eta^2s^{-2}\, dx  = (n-2)\int_{\{s > \eps\}\cap B_{2R}} u_s^2 s^{-2\alpha-2}\rho^2_R\, dx + (n-2)\int_{\{s\le \eps\}\cap B_{2R}} u_s^2 \eps^{-2\alpha}\rho_R^2. 
\]
On the other hand, putting it in the right hand side of \eqref{eq.stab2}, we get that
\begin{align*}
\int_{\R^n} u_s^2|\nabla \eta|^2\, dx \le \alpha^2\int_{\{s > \eps\}\cap B_{2R}} u_s^2 s^{-2\alpha-2}\rho_R^2\, dx + & \int_{B_{2R}\setminus B_R\cap \{s > \eps\}} u_s^2 s^{-2\alpha} |\nabla \rho_R|^2 \\
& + \eps^{-2\alpha}\int_{\{s \le \eps\}\}\cap B_{2R}}|\nabla\rho_R|^2\, dx.
\end{align*}

Putting all together, and using that $\int_{s \le \eps}|\nabla\rho_R|^2\, dx \le C R^{-2}\eps^{n-1} R$ we get 
\[
\begin{split}
&(n-2-\alpha^2)\int_{\{s > \eps\}\cap B_{2R}} u_s^2 s^{-2\alpha-2}\rho_R^2\, dx + (n-2)\int_{\{s\le \eps\}\cap B_{2R}} u_s^2 \eps^{-2\alpha}\rho_R^2 \le\\
&~~~~~~~~~~~~~~~~~~~~~~~~~~~~~~~~~~~~~~~~\le  R^{-2}\int_{B_{2R}\setminus B_R\cap \{s > \eps\}} u_s^2 s^{-2\alpha}  + C R^{-1}\eps^{n-1-2\alpha}.
\end{split}
\]

%Note now that, first changing variables $dx\mapsto s^{n-2} ds dt$, then writing in polar coordinates $(s, t) \mapsto (r \cos\theta, r\sin\theta)$, $ds\, dt \mapsto r\, dr\, d\theta$; and using that $u_s^2 \le C$ is bounded, we get that 
%\[
%\int_{B_{2R}\setminus B_R\cap \{s > \eps\}} u_s^2 s^{-2\alpha}\, dx\le C\int_{B_{2R}\setminus B_R} s^{ n-2-2\alpha}\, ds\,dt = C\int_R^{2R}\int_0^{2\pi} r^{n-1-2\alpha}\cos^{n-2-2\alpha}\theta\, dr\, d\theta.
%\]
Note now that, changing variables $dx\mapsto s^{n-2} ds dt$, and using that $u_s^2 \le |\nabla u|^2\le  C$ is bounded in $\R^n$, we get that 
\[
\int_{B_{2R}\setminus B_R\cap \{s > \eps\}} u_s^2 s^{-2\alpha}\, dx\le CR \int_\eps^{2R} s^{ n-2-2\alpha}\, ds = CR(R^{n-1-2\alpha} - \eps^{n-1-2\alpha})
\]

So that we get

\[
\begin{split}
&(n-2-\alpha^2)\int_{\{s > \eps\}\cap B_{2R}} u_s^2 s^{-2\alpha-2}\rho_R^2\, dx + (n-2)\int_{\{s\le \eps\}\cap B_{2R}} u_s^2 \eps^{-2\alpha}\rho_R^2 \le\\
&~~~~~~~~~~~~~~~~~~~~~~~~~~~~~~~~~~~~~~~~~~~~~~~~~~~~~~~~~~\le  C R^{n-2\alpha-2} + C R^{-1}\eps^{n-1-2\alpha}.
\end{split}
\]
Notice that $\int_{\{s\le \eps\}\cap B_{2R}} u_s^2 \eps^{-2\alpha}\rho_R^2\le C\eps^{-2\alpha} \int_{\{s\le \eps\}\cap B_{2R}} \, dx = CR\eps^{n-1-2\alpha}$. Thus, we have proved 
\[
(n-2-\alpha^2)\int_{\{s > \eps\}\cap B_{2R}} u_s^2 s^{-2\alpha-2}\rho_R^2\, dx  \le  C R^{n-2\alpha-2} + C R^{-1}\eps^{n-1-2\alpha} + CR\eps^{n-1-2\alpha} .
\]

In order to obtain $u_s\equiv 0$ in $\R^n$, we now want to take $\eps\downarrow 0$, $R\uparrow \infty$, and do that in such a way that the right-hand side above goes to zero. 

If $n-2> \alpha^2$, $n-2\alpha-2 < 0$ and $n-1 > 2\alpha$, the right-hand side goes to 0 as $R \to \infty$ if $\eps = R^{-\frac{1}{n-1-2\alpha-\eps_0}}\to 0$ (for some $\eps_0>0$ such that $n -1 > 2\alpha +\eps_0$). In particular, if such $\alpha$ exists, then $u_s\equiv 0$ and $u$ is one dimensional. 

Notice that if $n-2 > \alpha^2$ then $n-1 > \alpha^2+1\ge 2\alpha$, so that the two compatibility conditions become
\[
2\alpha > n -2 > \alpha^2.
\]
This is equivalent to 
\[
\left(\frac{n-2}{2}\right)^2 < n-2,
\]
or $2 < n < 6$. 

This means that the solution $u$, for $3\le n \le 5$, is one dimensional in the last coordinate, $u = u(t)$.
%NOTE: Case $n = 3$. One needs to choose test function appropriately so that potential irregularities at $s = 0$ disappear. This can be done ONLY if using $u_s \sim s$ near $s = 0$. 
\end{proof}

\section{Proof of Theorem~\ref{thm.1stab1phase}}
\label{sec.4}

Finally, we prove that stable, non-negative, axially symmetric solutions to the one-phase problem are also one-dimensional for $n\le 5$. Before doing so, let us remember an equivalent characterization of the stability condition for the one-phase problem as seen in \cite[Lemma 1]{CJK04}. 

\begin{lem}[\cite{CJK04}]
\label{lem.cjk}
Let $n \ge 2$, and let $u\in {\rm Lip} (\R^n)$ be a stable, non-negative solution to the one-phase problem; namely, a critical point of the  one-phase functional \eqref{eq.onephasefunctional} with non-negative second variation. Then, 
\begin{equation}
\label{eq.stabcond}
\int_{\partial\{u > 0\}} H\xi^2~d\sigma \le \int_{\{u > 0\}}|\nabla \xi|^2~dx
\end{equation}
for any $\xi \in C^\infty_0(\R^n)$, $\xi \ge 0$, and where $H\ge 0$ is the mean curvature of $\de\{u > 0\}$.
\end{lem}
\begin{proof}
It just follows as in \cite{CJK04}. Notice that \cite[Lemma 1]{CJK04} is stated for $1$-homogeneous energy minimizers in $n \ge 3$, but the proof actually shows that \ref{eq.stabcond} is equivalent to the non-negativity of the second variation of $J_0$. The $1$-homogeneity assumption is only used to argue that $|\nabla u |$ is globally bounded, which is an assumption in our case. And the $n \ge 3$ is used in the second part of the Lemma, but not in the first part, which is the one we will be using here. Therefore, the stability condition \eqref{eq.stabcond} also works in our case, for $n \ge 2$.
\end{proof}

\begin{rem}
The previous condition \eqref{eq.stabcond} is, in fact, an equivalent characterization of the non-negativity of the second variation of the one-phase functional \eqref{eq.onephasefunctional}. 
\end{rem}

\begin{proof}[Proof of Theorem~\ref{thm.1stab1phase}]
We separate the proof into two cases: $n = 2$ and $3\le n\le 5$. We start with the latter.
\\[0.2cm]
{\it \underline{\smash{Case 1}}, dimensions $3 \le n \le 5$:} By Lemma~\ref{lem.cjk} the stability condition can be equivalently written as 
\begin{equation}
\label{eq.stabcond2}
\int_{\partial\{u > 0\}} H\xi^2~d\sigma \le \int_{\{u > 0\}}|\nabla \xi|^2~dx
\end{equation}
for any $\xi \in C^\infty_0(\R^n)$, $\xi \ge 0$, and where $H\ge 0$ is the mean curvature of $\de\{u > 0\}$. By approximation, it will be enough to check for $\xi \in H^2_0(\R^n)$ and by taking $|\xi|$ instead of $\xi$ one can avoid the sign condition. 

The idea of the proof, now, is formally the same as in the proof of Theorem~\ref{thm.1stabprop} --- to take $\xi = u_s s^{-\alpha}$ as a test function in the stability condition \eqref{eq.stabcond2} --- however, we now need a more geometric approach to obtain the intermediate equalities. 

Choose $\xi = c\eta$ for some $c\in L^\infty$ smooth and $\eta\in C^2_c(\R^n)$ (with compact support). Then, integrating by parts, 
\begin{equation}
\label{eq.comp}
\begin{split}
\int_{\{u > 0\}}|\nabla (c\eta)|^2 & = \int_{\{u > 0\}}\left(|\nabla c|^2\eta^2+|\nabla \eta|^2 c^2\right)+\int_{\{u > 0\}} c \nabla(\eta^2)\cdot \nabla c\\
& = \int_{\{u > 0\}}\left(|\nabla \eta|^2 c^2-\eta^2 c\Delta c \right)+\int_{\partial\{u > 0\}} c \eta \nabla c\cdot \nu ~d\sigma,
\end{split}
\end{equation}
where $\nu$ denotes the outer unit normal of the set $\{u > 0\}$ at each point on the boundary $\partial \{u > 0 \}$. 

Let $u(x_1,\dots,x_n) = u(s, t)$ be axially symmetric, where $s = \sqrt{x_1^2+\dots+x_{n-1}^2}$ and $t = x_n$. In the set $\{ u > 0\}$ we know that $u$ is harmonic, with $\de_\nu u = -1$ on $\de\{u > 0\}$. In particular, $u$ is smooth up to the boundary, and $u_s = \de_s u$ is also smooth up to the boundary $\de\{u > 0\}$. Therefore, we can take $c = \tilde u_s$, where $\tilde u_s$ denotes any smooth extension of $u_s$ to the whole domain. Since we are only interested in the behaviour in $\{u > 0 \}$ we will write $c  =u_s$, understanding that $\nabla c$ at the boundary is actually $\nabla \tilde u_s$ and coincides with the limit of $\nabla u_s$ coming from the set $\{ u > 0\}$. 

Notice, also, that as in Theorem~\ref{thm.1stabprop} (see \eqref{eq.Lst}), we have 
\[
-\Delta u_s +(n-2)\frac{u_s}{s^2}= 0,\quad\textrm{ in } \{ u > 0\}
\]
since $u$ is harmonic in $\{u > 0 \}$. Putting all together in \eqref{eq.stabcond2} we reach 
\[
\int_{\partial\{u > 0\}} Hu_s^2\eta^2~d\sigma \le \int_{\{u > 0\}}\left(|\nabla \eta|^2 u_s^2-(n-2)\eta^2 \frac{u_s^2}{s^2} \right)+\int_{\partial\{u > 0\}} u_s \eta \nabla u_s\cdot \nu ~d\sigma.
\]

Now, we claim that 
\begin{equation}
\label{eq.claimH}
\nabla (u_s)\cdot \nu = Hu_s\quad\textrm{ on }\quad \de\{u > 0 \},
\end{equation}
so that we deduce
\[
(n-2)\int_{\{u > 0\}}\eta^2 u_s^2 s^{-2}\le \int_{\{u > 0\}}|\nabla \eta|^2 u_s^2.
\]
Noting that $u_s \equiv 0$ in $\{ u  = 0\}$ we have obtained the inequality \eqref{eq.stab2} from Theorem~\ref{thm.1stabprop}. From here, we can proceed exactly as in Theorem~\ref{thm.1stabprop} to get the desired result. 

To finish, let us now prove \eqref{eq.claimH}. On the one hand, notice that $\nabla u(x) = u_s \boldsymbol{e}_s(x) + u_t \boldsymbol{e}_t$, where $\boldsymbol{e}_s(x) := (x_1^2+\dots+x_{n-1}^2)^{-1/2}\left(x_1\boldsymbol{e}_1+\dots+x_{n-1}\boldsymbol{e}_{n-1}\right)$ and $\boldsymbol{e}_t := \boldsymbol{e}_n$; and that $\frac{\de}{\de s} = \boldsymbol{e}_s(x) \cdot \nabla$. Then a simple computation shows
\[
\de_s \nabla u(x) = u_{ss} \boldsymbol{e}_s(x)+u_{ts}\boldsymbol{e}_t = \nabla u_s(x);
\]
that is, derivatives in the $s$ direction commute with derivatives in the $\boldsymbol{e}_i$ direction for any $1\le i \le n$, since $u$ depends only on $(s, t)$. Now notice that 
\[
\nu \cdot\nabla u_s(x) = \nabla u \cdot \de_s \nabla u = \frac12 \de_s \left(|\nabla u|^2\right),
\]
where we are also using that $\nu = \nabla u$ on $\de \{u > 0\}$. 

Finally, on the other hand, we may assume that for a point $x_\circ \in \de\{ u > 0\}$, we have $\nu(x_\circ) = \boldsymbol{e}_1$; otherwise we can rotate our setting. Then, from $|\nabla u|^2 = u_1^2+\dots+u_{n}^2$ we easily get that $\nabla \left(|\nabla u|^2\right)(x_\circ) = 2u_1 u_{11} \boldsymbol{e}_1$, where we are using that $u_i = 0$ for $2\le i \le n$. Now, noticing that $u_1(x_\circ) = \de_\nu u(x_\circ) = -1$ and that $u_{11}(x_\circ) = u_{\nu \nu }(x_\circ) = -H(x_\circ)$ (see \cite[Lemma 1]{CJK04}) we obtain that $\nabla \left(|\nabla u|^2\right) = 2H\nu$ in general. This completes the proof, since
\[
\frac12 \de_s\left(|\nabla u|^2\right) =\frac12 \boldsymbol{e}_s \cdot \nabla \left(|\nabla u|^2\right) = H \boldsymbol{e}_s \cdot\nu = H\boldsymbol{e}_s \cdot\nabla u = Hu_s,
\]
as we wanted to see. 
\\[0.2cm]
{\it \underline{\smash{Case 2}}, dimension $n = 2$:} Notice that the proof of \eqref{eq.claimH} also holds if instead of $s$ we take any direction $x_i$ for $i = 1, 2$, so that we have that for any $\eta\in C^2_c(\R^n)$, 
\begin{equation}
\label{eq.comp1}
\int_{\partial \{u  > 0\}} Hu_i^2 \eta^2 \, d\sigma = \int_{\partial\{ u > 0\}} u_i\eta\nabla u_i\cdot \nu\, d\sigma.
\end{equation}
Moreover, from \eqref{eq.comp} and using that $\Delta u_i = 0$ in $\{u > 0\}$ we have that 
\begin{equation}
\label{eq.comp2}
\begin{split}
\int_{\partial\{ u > 0\}} u_i\eta\nabla u_i\cdot \nu\, d\sigma & = \int_{\{ u > 0\}} |\nabla (u_i\eta)|^2 - \int_{\{u > 0\}} u_i^2 |\nabla \eta|^2 \\
& = \int_{\{ u > 0\}} |\nabla u_i|^2\eta^2 + 2\int_{\{u > 0\}} u_i \eta\nabla u_i\cdot\nabla \eta\\
& = \int_{\{ u > 0\}} |\nabla u_i|^2\eta^2 + \frac12\int_{\{u > 0\}} \nabla u_i^2\cdot\nabla \eta^2
\end{split}
\end{equation}

Now, as in \cite[Second proof of Conjecture 1.1]{FV08} (which uses some ideas from \cite{Far02} and \cite{Ca10}) for the general semilinear case, we check the stability condition \eqref{eq.stabcond2} with $\xi = |\nabla u |\eta$,
\begin{equation}
\label{eq.comp3}
\sum_{i = 1}^n \int_{\partial\{u > 0\}} Hu_i^2\eta^2\, d\sigma \le \int_{\{u > 0\}}|\nabla \left(|\nabla u |\eta\right)|^2\, dx.
\end{equation}
Let us develop the right-hand side 
\[
\int_{\{u > 0\}}|\nabla \left(|\nabla u |\eta\right)|^2 = \int_{\{u > 0\}}\left\{|\nabla |\nabla u ||^2\eta^2 + |\nabla u|^2|\nabla \eta|^2+\frac12 \nabla \eta^2\cdot\nabla |\nabla u|^2 \right\}.
\]

Using this in \eqref{eq.comp3}, together with \eqref{eq.comp1}-\eqref{eq.comp2} yields
\[
\int_{\{u > 0\}} |\nabla u|^2|\nabla \eta|^2 \ge \int_{\{|\nabla u| > 0 \}}\left( \sum_{i = 1}^n |\nabla u_i|^2 - |\nabla |\nabla u ||^2 \right)\eta^2.
\]
Notice that the right-hand side is taken now in the set where $\nabla u \neq 0$. Now, one can finish exactly as in \cite{FV08} by a capacity argument. Indeed, one can write 
\[
\sum_{i = 1}^n |\nabla u_i|^2 - |\nabla |\nabla u ||^2 = |\nabla u|^2\mathcal{G}^2 +|\nabla_T|\nabla u||^2
\]
(see, for instance, \cite[Lemma 2.1]{SZ98}) where $\nabla_T$ denotes the tangential gradient along the level set of $u$, and $\mathcal{G}^2$ denotes the sum of the squares of the principal curvatures of the level set of $u$; which is well-defined, since the domain is $\nabla u  \neq 0$. Thus,  
\[
\int_{\{u > 0\}} |\nabla u|^2|\nabla \eta|^2 \ge \int_{\{|\nabla u| > 0 \}}\left( |\nabla u|^2\mathcal{G}^2 +|\nabla_T|\nabla u||^2 \right)\eta^2.
\]
Now, by taking 
\[
\eta(x) := 
\left\{
\begin{array}{cl}
1 & \textrm{ if } |x|< 1,\\
\frac{\log{R} - \log |x|}{\log R} & \textrm{ if }1 \le |x|<  R,\\
0 & \textrm{ if }|x| \ge R,
\end{array}
\right.
\]
one gets, by letting $R \to \infty$, that $\mathcal{G}^2 = \nabla_T|\nabla u|\equiv 0$. Indeed, we are using here that $\eta(x )\to 1$ as $R\to \infty$ for each $x\in \R^2$, and $\int_{\R^2}|\nabla\eta|^2 \to 0$ as $R\to \infty$. Thus, the level sets of $u$ are parallel hyper-planes, and $u$ is one-dimensional. 
\end{proof}

\begin{rem}
In the case $n  =2$, we have not used the hypothesis that $u$ is even in the first variable (coming from the axial symmetry).
\end{rem}

\end{document}